\documentclass[a4paper]{article}
\usepackage{amsmath, amsfonts}
\usepackage{amssymb, latexsym}
\usepackage{amsthm}
\usepackage{stmaryrd}

\usepackage[
colorlinks=true,       
linkcolor=blue,          
citecolor=blue,        
plainpages=false,      
pdfpagelabels,
]{hyperref}

\usepackage[shortlabels]{enumitem}
\setlist[enumerate]{leftmargin=*,align=left,labelindent=\parindent}

\newcommand{\Bin}{\left\{ 0,1 \right\}}

\DeclareMathOperator{\amp}{\,\&\,}
\DeclareMathOperator{\imp}{\,\rightarrow\,}
\DeclareMathOperator{\defeqiv}{\stackrel{\textup{def}}{\iff}}
\DeclareMathOperator{\defeql}{\stackrel{\textup{def}}{\  =\  }}
\newcommand{\dotminus}{\mathbin{\ooalign{\hss\raise.5ex\hbox{$\cdot$}\hss\crcr$-$}}}

\newcommand{\Nat}{\mathbb{N}}
\newcommand{\Rat}{\mathbb{Q}}
\newcommand{\Real}{\mathbb{R}}
\newcommand{\FSeq}{\Nat^{*}}
\newcommand{\BSeq}{\Bin^{*}}

\newcommand{\PBaire}{\Nat^{\Nat}}
\newcommand{\PCantor}{\Bin^{\Nat}}

\newcommand{\nil}{\langle\, \rangle}
\newcommand{\seq}[1]{\langle#1\rangle}

\newcommand{\lh}[1]{\lvert #1 \rvert}
\newcommand{\PiFT}{\textrm{\textup{$\Pi^{0}_{1}$-FAN}}}
\newcommand{\PiFTM}{\textrm{\textup{$\Pi^{0}_{1}$-FAN$_{\mathrm{M}}$}}}
\newcommand{\PiFTMStar}{\textrm{\textup{$\Pi^{0}_{1}$-FAN$_{\mathrm{M}}^{*}$}}}

\newcommand{\cFT}{\textrm{\textup{c-FAN}}}
\newcommand{\scFT}{\textrm{\textup{sc-FAN}}}
\newcommand{\scFTStar}{\textrm{\textup{sc-FAN}$^{*}$}}

\newcommand{\UCT}{{\textrm{\textup{UCT}}}}
\newcommand{\SC}{{\textrm{\textup{SC}}}}
\newcommand{\SCB}{\ensuremath{{\textrm{\textup{SC}}}_{\mathbf{B}}}}

\newcommand{\zero}{0^{\omega}}
 
\newcommand{\CZF}{\mathrm{CZF}}

\newtheorem{theorem}{Theorem}[section]
\newtheorem{proposition}[theorem]{Proposition}
\newtheorem{lemma}[theorem]{Lemma}
\newtheorem{corollary}[theorem]{Corollary}
\theoremstyle{definition}
\newtheorem{definition}[theorem]{Definition}
\theoremstyle{remark}
\newtheorem{remark}[theorem]{Remark}

\numberwithin{equation}{section}

\title{A continuity principle equivalent to the monotone $\Pi^{0}_{1}$ fan
theorem}

\author{
  Tatsuji Kawai\\[0.5em]
\normalsize Japan Advanced Institute of Science and Technology\\
\normalsize 1-1 Asahidai, Nomi, Ishikawa 923-1292, Japan\\
\normalsize\texttt{tatsuji.kawai@jaist.ac.jp}}

\date{}

\begin{document}
\maketitle
\begin{abstract}
The strong continuity principle reads ``every pointwise continuous
function from a complete separable metric space to a metric space is
uniformly continuous near each compact image.'' We show that this
principle is equivalent to the fan theorem for monotone $\Pi^{0}_{1}$
bars. We work in the context of constructive reverse mathematics.
\bigskip

\end{abstract}

\section{Introduction}\label{sec:Introduction}
In Bishop's constructive mathematics \cite{Bishop-67}, the statement
``every pointwise continuous function on a compact metric space
is uniformly continuous'' implies the fan theorem for detachable bars
\cite{DienerLoebSeqRealValonUInt}, which is recursively
false~\cite[Chapter 4, Section 7.6]{ConstMathI}. This led Bishop to
adopt uniform continuity, or more generally uniform continuity on each
compact subset, as the fundamental notion of continuity in his theory
of metric spaces. Call this notion \emph{Bishop continuity}.

Bishop continuity works well for locally compact metric spaces
with which much of Bishop's theory of metric spaces is concerned.
When we go beyond the context of locally compact metric spaces, however, 
it does not behave well as a composition of two Bishop continuous
functions need not be Bishop continuous (see e.g.\
Schuster \cite[Section 3]{Schuster:what_is_continuity_constructively}). 
For this reason, Bridges~\cite{BridgesConstFunctAnalysis} introduced
the following notion of continuity for functions between
metric spaces.\footnote{The notion is attributed to Bishop \cite[page
60]{TechniqueConstAnal}.}
\begin{definition}\label{def:StrongContinuous}
A subset $L$ of a metric space $X$ is a \emph{compact image} if there
exist a compact metric space $K$ and a uniformly continuous function
$f \colon K \to X$ such that $L = f[K]$. A function $f \colon  X \to Y$
between metric spaces is \emph{strongly continuous}\footnote{The
  terminology is due to Troelstra and {van Dalen} \cite[Chapter
  7]{ConstMathII}.} if for each compact image $L \subseteq X$ and
  each $\varepsilon > 0$, there exists a $\delta > 0$ such that for all $x,u
  \in X$
\[
  x \in L \amp d_{X}(x,u) < \delta \imp d_{Y}(f(x),f(u)) <
  \varepsilon.
\]
\end{definition}

In this paper, we are interested in the strength of strong
continuity compared with that of pointwise continuity.
In particular, our aim is to capture the difference between strong
continuity and pointwise continuity of functions on a complete
separable metric space by a variant of fan theorem in the spirit of
constructive reverse mathematics \cite{ConstRevMatheCompactness,RevMatheInBISH}.
To this end, we introduce the following continuity principle,
called \emph{strong continuity principle}:
\begin{description}
  \item[\SC]
    Every pointwise continuous function from a complete separable
    metric space to a metric space is strongly continuous.
\end{description}
The main result of this paper is that $\SC$ is equivalent to
the fan theorem for monotone $\Pi^{0}_{1}$ bars ($\PiFTM$).
We also introduce a new variant of fan theorem, called $\scFT$, and
show that it is equivalent to $\PiFTM$. 
The principle $\scFT$ is quite similar to the principle $\cFT$ due to
Berger~\cite{BergerUCandcFT} but it is based on a slightly weaker
notion of c-set, and hence is stronger than $\cFT$.

The significance of our results are as follows. First, it gives a
logical characterisation of $\SC$ in terms of fan theorem. Second,
compared with the previous works on $\PiFTM$ in constructive reverse
mathematics \cite{BizarrePropertyPiFanM,DienerLoebSeqRealValonUInt},
$\SC$ is a more natural continuity principle for $\PiFTM$ as it
relates pointwise continuity with a constructively well-known notion
of continuity. Third, the equivalence of $\SC$ and $\PiFTM$
sheds
light on the relation between $\PiFTM$ the Uniform Continuity Theorem:
\begin{description}
  \item[\UCT]
    Every pointwise continuous function from a compact
    metric space to a metric space is uniformly continuous.
\end{description}
It has been known that $\PiFTM$ implies $\UCT$
\cite[Theorem 2]{DienerLoebSeqRealValonUInt}; however since $\SC$ looks
much stronger than $\UCT$, it is more likely that $\PiFTM$ is strictly
stronger than $\UCT$. Finally, the notion of c-set that is used to
define $\scFT$ seems to clarify the difference between $\PiFTM$ and
$\cFT$.

We work in Bishop-style informal constructive
mathematics \cite{Bishop-67}; however it should be straightforward
to formalise our result in Aczel's constructive set theory $\CZF$ with
Countable Choice \cite{Aczel-Rathjen-Note}.

\subsubsection*{Related works}\label{sec:RelatedWorks}
Connections between various forms of fan theorems and continuity
principles are extensively studied in constructive reverse
mathematics;
see Diener and Loeb \cite[Section 9]{DienerLoebSeqRealValonUInt} for
an overview.
In particular, several people established equivalence
between $\PiFTM$ and some forms of continuity principle:
the uniform locally constancy of every locally constant function from
$\Bin^{\Nat}$ to $\Real$ (Berger and Bridges
\cite{BizarrePropertyPiFanM}) and the uniform equicontinuity of every
equicontinuous sequence of functions from $\PCantor$ to $\Real$ (Diener
and Loeb~\cite{DienerLoebSeqRealValonUInt}).
Note that $\PiFTM$ should not be confused with
the similar principle $\PiFT$ mentioned in
\cite{BergerFANandUC,BergerSeparationFan,LubarskySepFan}, where the
latter does not require that bars be monotone.\footnote{In
\cite{BizarrePropertyPiFanM,DienerLoebSeqRealValonUInt}, the
principle $\PiFTM$ is called the fan theorem for $\Pi^{0}_{1}$ bars,
but its definition requires that bars be monotone.} As far as we know,
it is still open whether $\PiFTM$ and $\PiFT$ are equivalent.

\subsubsection*{Notations}
We fix our notation.
The set of finite sequences of natural
numbers is denoted by $\FSeq$, and the set of finite binary sequences
is denoted by $\BSeq$. The letters $i,j,k,n,m,N,M$ range over $\Nat$,
and $a,b,c$ range over $\FSeq$.  Greek letters
$\alpha,\beta,\gamma,\delta$ range over the sequences $\PBaire$.
An element of $\FSeq$ of length $n$ is denoted by
$\seq{a_0,\dots,a_{n-1}}$, and the empty sequence is denoted by
$\nil$. The length of $a$ is denoted by $\lh{a}$, and the
concatenation of $a$ and $b$ is denoted by $a * b$.  The concatenation
of a finite sequence $a$ and a sequence $\alpha$ is denoted by
$a * \alpha$.  The initial segment of $\alpha$ of length $n$ is
denoted by $\overline{\alpha}n$.
We assume a fixed bijective coding of finite sequences where
the sequence coded by $n$ is denoted by $\hat{n}$.
We sometimes use $\lambda n.$
notation for a sequence. In particular, we put $\zero = \lambda n.
0$, the constant $0$ sequence.

\section{The fan theorem for  monotone
  \texorpdfstring{$\Pi^{0}_{1}$}{Pi01} bars}\label{sec:PiFTM}
We first recall some basic notions related to trees.
A \emph{tree} is an inhabited decidable subset of $\FSeq$
that is closed under initial segments.
A \emph{spread} is a tree $T$ such that
\[
   \forall a \in T  \exists n \, T(a * \seq{n}).
\]
Here, $T(a * \seq{n})$ means $a * \seq{n} \in T$.
Let $T$ be a spread.
A sequence $\alpha$ is a \emph{path} of $T$, written $\alpha
\in T$, if $\forall n \, T(\overline{\alpha}n)$.
The set of paths of $T$ is denoted by $K_{T}$, i.e.\
\[
  K_{T} \defeql \left\{ \alpha \in \PBaire \mid \alpha \in T \right\}.
\]
A subset $P$ of a spread $T$ is said to be
\begin{itemize}
  \item a \emph{bar} of $T$ if 
           $
           \forall \alpha \in T  \exists n \,
          P(\overline{\alpha}n);
           $
  \item a \emph{uniform bar} if there exists an
        $N$ such that
           $
           \forall \alpha \in T  \exists n \leq N \,
           P(\overline{\alpha}n);
           $
  \item \emph{$\Pi^{0}_{1}$} if there exists a sequence $\left(
    B_{n} \right)_{n \in \Nat}$ of decidable subsets $B_{n}$ of
    $T$ such that
      $
      P = \bigcap_{n \in \Nat} B_{n};
      $

  \item \emph{monotone} if $\forall a \in P \forall b\left[ a * b \in T
    \imp P(a*b) \right]$.
\end{itemize}
When we say that $P$ is a bar of $T$, we assume that $P$ is a subset of
$T$.  When we say that $P$ is a bar, this means that $P$ is a
bar of the \emph{universal spread} $\FSeq$.

A \emph{fan} is a spread $T$ such that
\[
   \forall a \in T \exists N
  \forall n \left[T(a * \seq{n}) \imp n \leq N \right].
\]
For each fan $T$, the principle $\PiFTM(T)$ reads:
\begin{description}
  \item[$\PiFTM(T)$] Every monotone $\Pi^{0}_{1}$  bar of $T$ is uniform.
\end{description}
\emph{The fan theorem for monotone $\Pi^{0}_{1}$ bars} is the principle
$\PiFTM(\BSeq)$, which is simply denoted by $\PiFTM$. 

We introduce an auxiliary fan theorem $\PiFTMStar(T)$ for each fan
$T$.  Given a fan $T$ and a bar $P$ (of the universal
spread $\FSeq$), we say that $P$ is \emph{uniform
with respect to $T$} if $P \cap T$ is a uniform bar of $T$.
Then, for each fan
$T$, the principle $\PiFTMStar(T)$ reads:
\begin{description}
  \item[$\PiFTMStar(T)$] Every monotone $\Pi^{0}_{1}$  bar is uniform
    with respect to $T$.
\end{description}

For the proof of the next proposition,
we recall the following construction (cf.\ Troelstra and van Dalen
\cite[Chapter 4, Lemma 1.4]{ConstMathI}).
Given a spread $T$, let $\Gamma_{T} \colon \FSeq \to T$ be a function
defined by
\begin{align}
  \begin{aligned}\label{eq:Gamma}
  \Gamma_{T}(\langle \rangle) &\defeql \langle \rangle,\\
  \Gamma_{T}(a * \langle n \rangle) &\defeql \Gamma_{T}(a) * \langle
  i(a,n)\rangle,
  \end{aligned}
\end{align}
where
\begin{align*}
  i(a,n) \defeql
  \begin{cases}
    n
    &
    \text{if $\Gamma_{T}(a) * \langle n \rangle \in T$},  \\
    \min \left\{ m \mid \Gamma_{T}(a) * \langle m \rangle \in T \right\}
    & \text{otherwise}.
  \end{cases}
\end{align*}
Then, we extend $\Gamma_T$ to $\Gamma_T^{*} \colon \PBaire \to K_{T}$ by
\begin{equation}\label{eq:GammaExt}
  \Gamma_{T}^{*}(\alpha)(n) \defeql \Gamma_{T}(\overline{\alpha}(n+1))(n).
\end{equation}
Note that $\Gamma_T^{*}(\alpha) = \alpha$
whenever $\alpha \in T$.

\begin{proposition}\label{prop:FanMEquivFanMStar}
  For each fan $T$, $\PiFTM(T)$ is equivalent to $\PiFTMStar(T)$.
\end{proposition}
\begin{proof}
  Fix a fan $T$. First assume $\PiFTM(T)$. 
  Let $P \subseteq \FSeq$ be a monotone $\Pi^{0}_{1}$ bar.
  Then, $Q \defeql P \cap T$ is a monotone $\Pi^{0}_{1}$ bar of $T$,
  and hence it is uniform. Then 
  $P$ is uniform with respect to $T$.

  Conversely, assume $\PiFTMStar(T)$. Let $P$ be
  a monotone $\Pi^{0}_{1}$ bar of $T$.
  Define  $Q \subseteq \FSeq$ by
  \[
    Q(a) \defeqiv P(\Gamma_T(a)).
  \]
  Then, $Q$ is a monotone $\Pi^{0}_{1}$ subset of $\FSeq$, which is easily
  seen to be a bar by the definition of 
  $\Gamma_T^{*}$ in \eqref{eq:GammaExt}. Hence, $Q$ is uniform with
  respect to $T$. Then, $P$ is uniform.
\end{proof}

\begin{lemma}\label{lem:PiBFTimpPiFT}
  $\PiFTM$ implies $\PiFTM(T)$ for all fan $T$.
\end{lemma}
\begin{proof}
The proof is essentially the same as the one given by Troelstra and {van Dalen} \cite[Chapter 4, Section
7.5]{ConstMathI} for a similar fact about fan theorem. 

  First, suppose that $T$ is a subfan of $\BSeq$.
  Let $\Delta_{T} \colon \BSeq \to T$ be the function that is similarly
  defined as $\Gamma_{T}$ (cf.\ \eqref{eq:Gamma}), and let
  $\Delta^{*}_{T} \colon \PCantor \to K_{T}$ be the extension of
  $\Delta_{T}$ that
  is similarly defined as $\Gamma_{T}^{*}$ (cf.\ \eqref{eq:GammaExt}).
  Then, it is straightforward to show that $\PiFTM$ implies $\PiFTM(T)$.

  Next, we consider the case where $T$ is not necessarily a
  subfan of $\BSeq$. In this case, we can embed $T$ into  $\BSeq$ by a function
  $\Phi$ given by
  \begin{align*}
    \Phi(\langle \rangle) &\defeql \langle \rangle,\\
    \Phi(a * \langle i \rangle) &\defeql \Phi(a) *
    \langle 0 \rangle * \langle \underbrace{1,\dots,1}_{i + 1}
    \rangle.
  \end{align*} 
  Let $T'$ be a subfan of $\BSeq$ obtained by closing the image of
  $\Phi$ under initial segments. Then, the function $\Phi$ is naturally extended
  to a uniformly continuous function $\Phi^{*} \colon K_{T} \to K_{T'}$ as
  we did for $\Gamma_T$ (cf.\ \eqref{eq:GammaExt}).
  Let $\Psi^{*} \colon K_{T'} \to K_{T}$ be the uniformly continuous
  inverse of $\Phi$ that is defined as the extension of a function  $\Psi \colon
  T' \to T$ defined by
  \begin{align*}
    \Psi(\langle \rangle) = \Psi(\langle 0 \rangle)
    &\defeql \langle \rangle,\\
    \Psi(a * \langle 0 \rangle * \langle \underbrace{1,\dots,1}_{n + 1} \rangle)
    &\defeql \Psi(a * \langle 0 \rangle), \\
    \Psi(a * \langle 0 \rangle * \langle \underbrace{1,\dots,1}_{n + 1} \rangle
    * \langle 0 \rangle)
    &\defeql \Psi(a * \langle 0 \rangle) * \langle n \rangle.
  \end{align*} 
  The modulus of $\Psi^{*} \colon K_{T'} \to K_{T}$
  is provided by a function $\phi \colon \Nat \to \Nat$ defined by 
  \[
    \phi(n) \defeql \left(\sum_{i = 0}^{n} \max\Bigl\{ a(i-1) \mid a
    \in T \amp \lh{a} = i \Bigr\} \right) + 2n + 1.
  \]
  Indeed, for each $\alpha, \beta \in K_{T'}$ and $n \in \Nat$, we have
  \[
    \overline{\alpha}\phi(n) = \overline{\beta}\phi(n)
    \imp
    \overline{\Psi^{*}(\alpha)}n = \overline{\Psi(\overline{\alpha}\phi(n))}n
    = \overline{\Psi(\overline{\beta}\phi(n))}n =
    \overline{\Psi^{*}(\beta)}n.
  \]
  Now, assume $\PiFTM(T')$, and let $P$ be a
  monotone $\Pi^{0}_{1}$ bar of $T$. Then
    $
    \forall \alpha \in T' \exists n \, P(\overline{\Psi^{*}(\alpha)}n),
    $
    so that 
    $
    \forall \alpha \in T' \exists m \, P(\Psi(\overline{\alpha}m)).
    $
  Thus, a subset $P' \subseteq T'$ defined by
  \[
    P'(a) \defeql P(\Psi(a))
  \]
  is a monotone $\Pi^{0}_{1}$ bar of $T'$.  By $\PiFTM(T')$, there
  exists an $N$ such that
  $\forall \alpha \in T' \, P(\Psi(\overline{\alpha}N))$.
  Define $M$ by 
  \[
    M \defeql \max \left\{ |\Psi(a)| \mid a \in T' \amp |a| = N \right\}.
  \]
  Then for each $\alpha \in T$, we have
  $P(\Psi(\overline{\Phi^{*}(\alpha)}N))$. Since $\Psi(\overline{\Phi^{*}(\alpha)}N)$ 
  is an initial segment of $\alpha$ and
  $|\Psi(\overline{\Phi(\alpha)}N)| \leq M$, we have
  $P(\overline{\alpha}M)$. Hence $P$ is a uniform bar of $T$.
  Therefore $\PiFTM(T')$ implies  $\PiFTM(T)$.
  
  Combining the two arguments, we see that $\PiFTM$ implies $\PiFTM(T)$.
\end{proof}

\begin{corollary}\label{cor:PiFanSummary}
 The following are equivalent.
 \begin{enumerate}
   \item $\PiFTM$.
   \item $\PiFTM(T)$ for all fan $T$.
   \item $\PiFTMStar(T)$ for all fan $T$.
 \end{enumerate}
\end{corollary}

\section{The principle \texorpdfstring{$\scFT$}{sc-FAN}}\label{sec:scFT}
We introduce another fan theorem $\scFT$ based on the notion of c-set.

Let $T$ be a spread. A subset $P \subseteq T$ is a \emph{c-set} if there exists a
decidable subset $D \subseteq \FSeq$ such that
\begin{equation}\label{def:cset}
  P(a) \leftrightarrow \forall b \left(  a * b \in D\right)
\end{equation}
for all $a \in T$.
Note that every c-set is \emph{monotone}.
A \emph{c-bar} of $T$ is a c-set that is a bar of $T$.

For each fan $T$, the principle $\scFT(T)$ reads:
\begin{description}
 \item[$\scFT(T)$] Every c-bar of $T$ is uniform.
\end{description}
\emph{The strong continuous fan theorem} is the principle
$\PiFTM(\BSeq)$, which is simply denoted by $\scFT$. 
\begin{remark}
The principle $\scFT$ should not be confused with $\cFT$
introduced in \cite{BergerUCandcFT}, since we have slightly weaken
the notion of c-set from the one that is used to define $\cFT$.
In \cite{BergerUCandcFT}, a c-set is defined as in \eqref{def:cset}
except that the quantification $\forall b$ is restricted to
$T$.\footnote{In \cite{BergerUCandcFT}, the fan $T$ is fixed to the
binary tree.}
\end{remark}

%
%

%
%
\begin{lemma}
  \label{lem:PiMBarContainsCBar}
  Let $T$ be a spread and $P \subseteq T$ be a monotone $\Pi_{0}^{1}$ 
  bar of $T$. Then there exists a c-bar $Q$ of $T$ such that
  $Q \subseteq P$.
\end{lemma}
\begin{proof}
  Let $\left( B_{n} \right)_{n \in \Nat}$ be a sequence of decidable subsets of
  $T$ such that $P = \bigcap_{n \in \Nat} B_{n}$. Define $D \subseteq \FSeq$
  inductively
  by
  \begin{align*}
    \begin{cases}
    \nil \in D,&\\
    a * \seq{n} \in D & \text{if $B_{n}(\Gamma_T(a))$},
    \end{cases}
  \end{align*}
  where $\Gamma_T \colon \FSeq \to T$ is given in \eqref{eq:Gamma}.
  Define a c-set $Q \subseteq T$ by
  \[
    Q(a) \defeqiv \forall b\, D(a * b).
  \]
  Since $P(a) \leftrightarrow \forall n D(a * \seq{n})$ for each $a
  \in T$, we have $Q \subseteq P$. To see that $Q$ is bar of $T$, let $\alpha
  \in T$.  Since $P$ is a bar, there exists an $n$ such that
  $P(\overline{\alpha}n)$. Sicne $P$ is monotone, we also have
  $P(\overline{\alpha}(n+1))$. We claim that
  $Q(\overline{\alpha}(n+1))$, i.e. $\forall b
  D(\overline{\alpha}(n+1) * b)$, which is proved by induction on the
  length of $b$. If $\lh{b} = 0$, then $D(\overline{\alpha}(n+1))
  \leftrightarrow B_{\alpha_{n}}(\overline{\alpha}n)$, and the right
  hand side follows from the fact that $P(\overline{\alpha}n)$.  The
  inductive case follows from the monotonicity of $P$.
\end{proof}

\begin{proposition}\label{prop:PiFTMscFT}
  For each fan $T$, $\PiFTM(T)$ is equivalent to  $\scFT(T)$.
\end{proposition}
\begin{proof}
Fix a fan $T$. First assume $\PiFTM(T)$.
Let $Q \subseteq T$ be a c-set, and $D \subseteq \FSeq$ be a decidable
subset such that $Q(a) \leftrightarrow \forall b D(a*b)$ for all
$a \in T$. For each $n$, define $B_{n} \subseteq T$ by 
\[
  B_{n}(a) \defeqiv D(a * \hat{n}).
\]
Then $Q \defeql \bigcap_{n \in \Nat} B_{n}$. By
$\PiFTM(T)$, we conclude that $Q$ is a uniform.

Conversely, assume $\scFT(T)$, and let $P$ be a
monotone $\Pi^{0}_{1}$ bar of $T$. By Lemma
\ref{lem:PiMBarContainsCBar}, there exists a c-bar of $T$
that is contained in $P$.
By $\scFT(T)$,  we conclude that $P$ is uniform.
\end{proof}
We introduce an auxiliary fan theorem $\scFTStar(T)$ for each fan
$T$, which reads:
\begin{description}
  \item[$\scFTStar(T)$] 
    Every c-bar of $\FSeq$ is uniform with respect to $T$.
\end{description}
It is easy to see that Proposition \ref{prop:PiFTMscFT} holds for
the starred variants as well.
\begin{proposition}\label{prop:PiFTMStarscFTStar}
  For each fan $T$, $\PiFTMStar(T)$ is equivalent to  $\scFTStar(T)$.
\end{proposition}

By Proposition \ref{prop:PiFTMscFT}, Proposition
\ref{prop:PiFTMStarscFTStar}, and Corollary
\ref{cor:PiFanSummary}, we have the following.
\begin{proposition}\label{prop:PiFMEquivscFT}
  The following are equivalent:
  \begin{enumerate}
    \item  $\PiFTM$.
    \item  $\scFT$.
    \item  $\scFT(T)$ for all fan $T$.
    \item  $\scFTStar(T)$ for all fan $T$.
  \end{enumerate}
\end{proposition}

\section{Strong continuity principle}
\label{sec:StrongContinuityPrinciple}
\subsection{Strongly continuous functions}
\label{sec:StrongContinuity}
We give some convenient characterisations
of strongly continuous functions on complete metric spaces (cf.\
Definition \ref{def:StrongContinuous}).

Let $f \colon  X \to Y$ be a function between metric spaces,
and let $K \subseteq X$ be a compact subset. We say that $f$ is 
\emph{uniformly continuous near $K$} if 
for each $\varepsilon > 0$, there exists a $\delta > 0$ such that
for all $x,u \in X$
\[
  x \in K \amp d_{X}(x,u) < \delta \imp d_{Y}(f(x),f(u)) <
  \varepsilon.
\]
Since a closed subset of a complete metric space is complete,
a strongly continuous function on a complete metric space 
admits a simple characterisation.
\begin{lemma}
A function from a complete metric space to a
metric space is strongly continuous if and only if 
it is uniformly continuous near each
compact subset of its domain.
\end{lemma}
Next, we characterise a strongly continuous function on Baire space
$\PBaire$.
Note that Baire space is a complete separable metric space
with the product metric given by 
\begin{equation*}
  d(\alpha,\beta)
  \defeql \inf \left\{ 2^{-n} \mid \overline{\alpha} n =
  \overline{\beta} n\right\}.
\end{equation*}
Let $f \colon \PBaire \to X$ be a function to a metric space $(X,d)$,
and let $T$ be a fan. We say that $f$ is \emph{uniformly continuous
near $T$} if for each $\varepsilon > 0$, there exists an $N$ such that
  \begin{equation*}
    \forall \alpha \in T 
    \forall \beta 
    \left[ 
    \overline{\alpha} N = \overline{\beta} N 
  \imp d(f(\alpha), f(\beta)) < \varepsilon \right].
  \end{equation*}

It is well known that each inhabited compact subset of Baire space can be represented by the set of paths of some fan.  Then, the following is clear.
\begin{lemma}\label{lem:StrongContBaire}\leavevmode
  A function $f \colon \PBaire \to X$  to a metric space $X$ is strongly
  continuous if and only if $f$ is uniformly continuous near each fan.
\end{lemma} 
In particular, a function $f \colon \PBaire \to \Nat$ to the discrete
space of natural numbers is strongly continuous if and only if
for each fan $T$, there exists an $N$ such that
  \begin{equation*}
    \forall \alpha \in T 
    \forall \beta 
    \left[ 
    \overline{\alpha} N = \overline{\beta} N 
  \imp f(\alpha) = f(\beta) \right].
  \end{equation*}

\subsection{Strong continuity principle for Baire space}
We focus on the following special case of $\SC$ (cf.\
\ref{sec:Introduction} Introduction):
\begin{description}
  \item[\SCB] Every pointwise continuous function
    $f \colon \PBaire \to \Nat$ is strongly continuous.
\end{description}
For each fan $T$, we also introduce a local version:
\begin{description}
  \item[\SCB($T$)] Every pointwise continuous function
    $f \colon \PBaire \to \Nat$ is uniformly continuous near $T$.
\end{description}

\begin{proposition}\label{prop:PiFTMEquivSC}
  For each fan $T$, $\SCB(T)$ is equivalent to $\PiFTM(T)$.
\end{proposition}
\begin{proof} 
Fixed a fan $T$.
First, assume $\SCB(T)$. By Proposition \ref{prop:FanMEquivFanMStar}
and Proposition \ref{prop:PiFTMStarscFTStar}, it suffices to show that $\scFTStar(T)$
holds.
Let $P$ be a c-bar of $\FSeq$, and let $D \subseteq \FSeq$ be a
decidable subset such that
\[
  P(a) \leftrightarrow \forall b D(a*b).
\]
Define a function $f \colon \PBaire \to \Nat$ by
\begin{equation*}
  f(\alpha) \defeql \max D_{\alpha},
\end{equation*}
where $D_{\alpha}$ is a bounded subset of $\Nat$ given by
\[
  D_{\alpha}
  \defeql
  \left\{ n  \mid \neg D(\overline{\alpha}n) \right\}
  \cup \{1\}.
\]
Since $P$ is a c-bar of $\FSeq$, the function $f$ is pointwise
continuous. Hence $f$ is uniformly
continuous near $T$ by $\SCB(T)$. Thus, there exists an $N$ such that 
\[
  \forall \alpha \in T
  \forall \beta \,
  \left[ \overline{\alpha} N = \overline{\beta} N \imp f(\alpha) =
  f(\beta) \right].  
\]
Define $M$ by
\[
  M \defeql \max \left\{N, \max \left\{ f(a * \zero) \mid a \in T \amp
  \lh{a} = N \right\}  \right\} + 1.
\]
We show that $ \forall \alpha \in T \, P(\overline{\alpha}M)$.
Let $\alpha \in T$, and put $a \defeql \overline{\alpha}M$.  We must show
that $ \forall b \, D(a * b)$.
But for any $b \in \FSeq$, we have 
$M > f(a * \zero) = f(a * b * \zero)$, so  we must have $D(a * b)$. Hence $P$ is a uniform bar of $T$.

Conversely, assume $\PiFTM(T)$.
Let $f \colon \PBaire \to \Nat$ be a pointwise continuous function.
For each $n$, define $B_{n} \subseteq T$ by
\[
  B_{n}(a) \defeqiv f(a * \zero) = f(a * \hat{n} * \zero).
\]
Put $P \defeql \bigcap_{n \in \Nat} B_{n}$.
Clearly, $P$ is a monotone $\Pi^{0}_{1}$ subset of $T$,
and since $f$ is pointwise continuous, $P$ is a bar of $T$.
By $\PiFTM(T)$, there exists an $N$ such that
  $
   \forall \alpha \in T \, P(\overline{\alpha}N).
  $
Then
\[
   \forall \alpha \in T  \forall b \,
   \left[ 
   f(\overline{\alpha}N * \zero) = f(\overline{\alpha}N * b * \zero) \right].
\]
Let $\alpha \in T$ and $\beta \in \PBaire$ and suppose that
$\overline{\alpha}N = \overline{\beta}N$. 
Since $f$ is pointwise continuous, there exists an $m \geq N$ such
that $f(\overline{\alpha}m * \zero) = f(\alpha)$ and
$f(\overline{\beta}m * \zero) = f(\beta)$. Then,
\[
  f(\alpha)
  = f(\overline{\alpha}m * \zero)
  = f(\overline{\alpha}N * \zero)
  = f(\overline{\beta}m * \zero)
  = f(\beta).
\]
Hence, $f$ is uniformly continuous near $T$.
\end{proof}

By Corollary \ref{cor:PiFanSummary}, we have the following.
\begin{proposition}\label{prop:PiFTMequivSC}
$\SCB$ is equivalent to $\PiFTM$.
\end{proposition}

\subsection{Main result}\label{sec:Main}
Our main aim is to show the equivalence of $\SC$ and $\PiFTM$, which 
is mediated by the following equivalence.
The main idea of its proof is due to Diener and Loeb [9,
Theorem 2], where they showed that $\PiFTM$ implies $\UCT$.
\begin{proposition}\label{prop:PiFTMEquivBtoX}
  The following statements are equivalent:
 \begin{enumerate}
   \item\label{prop:PiFTMEquivBtoX1} $\PiFTM$.
   \item\label{prop:PiFTMEquivBtoX2} 
    Every pointwise continuous function from Baire space to a metric
    space is strongly continuous.
 \end{enumerate}
\end{proposition}
\begin{proof} 
  By Proposition \ref{prop:PiFTMequivSC}, it suffices to show
  that \ref{prop:PiFTMEquivBtoX1} implies \ref{prop:PiFTMEquivBtoX2}.
 
  Assume $\PiFTM$. Let $f \colon \PBaire \to X$ be a pointwise
  continuous function to a metric space $(X,d)$, and let $T$ be a fan.
  We must show that $f$ is uniformly continuous near $T$.
  Let $\varepsilon > 0$.  Using Countable Choice, construct a function
  $\lambda \colon \FSeq \times \FSeq \to \Bin$ such that
  \begin{align*}
    \lambda(a,b) = 1 &\implies d(f(a * \zero),f(b * \zero))
    < \varepsilon/3,\\
    \lambda(a,b) = 0 &\implies d(f(a * \zero),f(b * \zero))
    > \varepsilon / 6.
  \end{align*}
  Define $P \subseteq T$ by
  \[
    P(a) \defeqiv \forall b,c \left[ \lambda(a*b,a*c) = 1  \right].
  \]
  Clearly, $P$ is a monotone $\Pi^{0}_{1}$ subset of $T$.
  To see that $P$ is a bar of $T$, let $\alpha \in T$. Since
  $f$ is pointwise continuous, there exists an $N$ such that 
  \[
    \forall a,b \left[ d(f(\overline{\alpha}N * a * \zero),
    f(\overline{\alpha}N * b * \zero)) < \varepsilon/6 \right].
  \]
  Hence $P(\overline{\alpha}N)$.
  By $\PiFTM$ and Corollary \ref{cor:PiFanSummary}, there exists an $N$ such that $\forall \alpha \in T
  P(\overline{\alpha}N)$. Thus, 
  \[
    \forall \alpha \in T
    \forall a,b \left[ d(f(\overline{\alpha}N * a * \zero),
    f(\overline{\alpha}N * b * \zero)) < \varepsilon/3  \right].
  \]
  Let $\alpha \in T$ and $\beta \in \PBaire$ and suppose that
  $\overline{\alpha}N = \overline{\beta}N$. 
  Since $f$ is pointwise continuous, there exists an $m \geq N$ such
  that $d(f(\overline{\alpha}m * \zero), f(\alpha)) <
  \varepsilon/3$ and
  $d(f(\overline{\beta}m * \zero), f(\beta)) < \varepsilon/3$. Then,
  \begin{align*}
    &d(f(\alpha), f(\beta))\\
    &\leq
    d(f(\alpha), f(\overline{\alpha}m * \zero))
    + d(f(\overline{\alpha}m * \zero), f(\overline{\beta}m * \zero))
    + d(f(\overline{\beta}m * \zero), f(\beta))\\
    &<
    \varepsilon/3 + \varepsilon/3 + \varepsilon/3 = \varepsilon.
  \end{align*}
  Therefore, $f$ is uniformly continuous near $T$.
\end{proof}

For the next proposition, we use a combination of spread representations
of a complete separable metric space and a compact metric space (cf.\
Troelstra and van Dalen \cite[Chapter 7, Section 2.3 and
  Section 4.2]{ConstMathII}). 

Let $(X,d)$ be a complete separable metric space together with a dense
sequence $(p_{n})_{n \in \Nat}$ in $X$, and let $K \subseteq X$ be an
inhabited compact subset (i.e.\ a complete and totally bounded subset).
For each $k$, we can find indices
$\left\{i_{0},\dots,i_{n_k} \right\} \subseteq \Nat$ such that
\begin{align}
\begin{aligned}\label{eq:2knet}
  &\forall x \in K \exists j \leq n_k \, d(x,p_{i_j}) < 2^{-k},\\
  &\forall j \leq n_k \exists x \in K  \, d(x,p_{i_j}) < 2^{-k}.
\end{aligned}
\end{align}
By Countable Choice, we have a function $I$ with the domain $\Nat$
that assigns to each $k$ a  finite indices with the property
\eqref{eq:2knet}.
By Countable Choice again, we also have a function $\alpha_{\Rat}
\colon \Nat \times \Nat \times \Nat \to \Rat$ such that
\[
  \lh{d(p_{i},p_{j}) - \alpha_{\Rat}(i,j,k)} < 2^{-k}
\]
for all $i,j, k$. Then, for each $k$ and $i$, the subset
\[
  \left\{ j \in \Nat \mid \alpha_{\Rat}(i,j,k+1) < 2^{-k+1}\right\}
\]
is decidable and inhabited by $i$. Moreover, for each $k$ and $i \in I_{k+1}$, 
there exists an $x \in K$ such that $ d(p_{i},x) < 2^{-(k+1)} $.
Then, there exists a $j \in I_{k+2}$ such that $d(x,p_j) < 2^{-(k+2)}$.
Thus
\begin{align*}
  \alpha_{\Rat}(i,j,k+1) 
  &< 2^{-(k+1)} + d(p_i,p_j)\\
  &\leq 2^{-(k+1)} + d(p_i,x) + d(x, p_j)\\
  &< 2^{-(k+1)} + 2^{-(k+1)} + 2^{-(k+2)} \\
  &< 2^{-k+1}.
\end{align*}
Hence, for each $k$ and $i \in I_{k+1}$, the subset
\[
  \left\{ j \in I_{k+2} \mid \alpha_{\Rat}(i,j,k+1) < 2^{-k+1}\right\}
\]
is finite and inhabited.

We now define a spread $T_{0}$ and a fan $T_{1}$ by specifying their
branches:
\begin{align*}
  \gamma \in T_{0}
  &\defeqiv
  \forall k \left[ \alpha_{\Rat}(\gamma_{k}, \gamma_{k+1}, k+1) <
  2^{-k+1}\right],\\
  \gamma \in T_{1}
  &\defeqiv
  \forall k \left[ \gamma_{k} \in I_{k+1} \amp
    \alpha_{\Rat}(\gamma_{k}, \gamma_{k+1}, k+1) <
  2^{-k+1}\right].
\end{align*}
Note that $T_{1}$ is a subfan of $T_{0}$.
\begin{lemma}
    \label{lem:StdRep}\leavevmode
  \begin{enumerate}
    \item\label{lem:StdRep1} For each $\gamma \in T_{0}$, $(p_{\gamma_{k}})_{k \in \Nat}$
      is a Cauchy sequence.
    \item\label{lem:StdRep2} $\forall k \left[ d(x,p_{\gamma_{k}}) < 2^{-(k+1)} \right]
      \imp \gamma \in T_{0}$ for each $x \in X$ and $\gamma$.
    \item\label{lem:StdRep3} $\forall k \left[ d(x,p_{\gamma_{k}}) < 2^{-(k+1)}  \amp
      \gamma_{k} \in I_{k+1} \right]
      \imp \gamma \in T_{1}$ for each $x \in X$ and $\gamma$.
    \item\label{lem:StdRep4} For each $\gamma \in T_{0}$, write $x_{\gamma} \defeql \lim_{k
      \to \infty} p_{\gamma_{k}}$.
      Then
      \[
        \forall x \in K \exists \gamma \in T_{1}
        \left[ x = x_{\gamma} \amp \forall k \left[ U(2^{-(k+1)},x)
        \subseteq V_{\overline{\gamma}k}\right]\right],
      \]
      where
      \begin{align*}
        U(2^{-(k+1)},x) &\defeql \left\{ y \in X \mid d(x,y) <
          2^{-(k+1)} \right\}, \\
          V_{\overline{\gamma}k} &\defeql \left\{ x_{\beta} \mid \beta \in
          T_{0} \amp \overline{\beta}k = \overline{\gamma}k \right\}.
        \end{align*}
  \end{enumerate}
\end{lemma}
\begin{proof}
We closely follow the proof by Troelstra and van Dalen
\cite[Chapter 7, Proposition 2.4]{ConstMathII}.
\ref{lem:StdRep1} -- \ref{lem:StdRep3} are
straightforward. 
As for \ref{lem:StdRep4},
let $x \in K$. By Countable Choice, we can find a sequence $\gamma$ such that
$ \forall k \left[ d(x,p_{\gamma_{k}}) < 2^{-(k+1)} \amp
\gamma_{k} \in I_{k+1}\right]$.
Then $\gamma \in T_{1}$ by \ref{lem:StdRep3}, and $x = \lim_{k \to
\infty} p_{\gamma_{k}} = x_{\gamma}$.
Now fix $k$ and $y \in U(2^{-(k+1)},x)$. By Countable Choice,
construct a sequence $\beta$ such that
$ \forall n \left[ d(y,p_{\beta_{n}}) < 2^{-(n+1)} \right]$.
Then $\beta \in T_{0}$ by \ref{lem:StdRep2} and $y = x_{\beta}$.
Since
\begin{align*}
  d(p_{\gamma_{k}},p_{\beta_{k+1}})
  &\leq d(p_{\gamma_{k}}, x) + d(x,y) + d(y, p_{\beta_{k+1}})\\
  &< 2^{-(k+1)} + 2^{-(k+1)} + 2^{-(k+1)} 
  = 3 \cdot 2^{-(k+1)},
\end{align*}
we have $\alpha_{\Rat}(\gamma_{k},\beta_{k+1}, k+1) < 3 \cdot
2^{-(k+1)} + 2^{-(k+1)} = 2^{-k+1}$. This implies that the sequence $\delta =
\overline{\gamma}(k+1) * \lambda n. \beta(k+1+n)$ is in $T_{0}$.
Since $y = x_{\delta}$, we see that $y \in V_{\overline{\gamma}k}$.
Therefore $U(2^{-(k+1)},x) \subseteq V_{\overline{\gamma}k}$.
\end{proof}
Let $\varphi_{T_{0}} \colon K_{T_{0}} \to X$ be the mapping $\gamma
\mapsto x_{\gamma}$. Note that $\varphi_{T_{0}}$ is uniformly continuous.
We also have a mapping $\Gamma_{T_{0}} \colon \FSeq \to
T_{0}$ defined by \eqref{eq:Gamma}.
Then, $\Gamma_{T_0}$ can be naturally extended to a uniformly continuous
function $\Gamma^{*}_{T_0} \colon \PBaire \to K_{T_{0}}$ that is identity on
$T_{0}$ (cf.\ \eqref{eq:GammaExt}).
Write $\Phi_{T_{0},X} \colon \PBaire \to X$ for the composition
$\varphi_{T_{0}} \circ \Gamma^{*}_{T_{0}}$.

\begin{proposition}\label{prop:BtoXEquivCSMtoX}
  The following statements are equivalent:
 \begin{enumerate}
   \item\label{prop:BtoXEquivCSMtoX1}
    Every pointwise continuous function from Baire space to a metric
    space is strongly continuous.
   \item\label{prop:BtoXEquivCSMtoX2} 
     \textup{(\SC)} Every pointwise continuous function from a
     complete separable metric space to a metric space is strongly
     continuous.
 \end{enumerate}
\end{proposition}
\begin{proof}
  It suffices to show that
  \ref{prop:BtoXEquivCSMtoX1} implies \ref{prop:BtoXEquivCSMtoX2}.
  Assume \ref{prop:BtoXEquivCSMtoX1}. Let $f \colon X \to Y$ be a
  pointwise continuous function from a complete separable metric space
  $X$ to a metric space $Y$, and let $K \subseteq X$ be an inhabited compact
  subset. We must show that $f$ is uniformly continuous near $K$.

  Construct a spread $T_{0}$ and a subfan $T_{1}$ of $T_{0}$ with
  the property described in Lemma \ref{lem:StdRep}.
  Let $\Phi_{T_{0},X} \colon \PBaire \to X$ be the uniformly
  continuous function that is introduced just above this proposition.
  Then, $g \defeql f \circ \Phi_{T_{0},X}$ is pointwise continuous,
  and thus strongly continuous by the assumption. Now fix
  $\varepsilon > 0$. Since $g$ is uniformly continuous near
  $T_{1}$, there exists an $N$ such that
  \[
    \forall \alpha \in T_{1} \forall \beta
    \left[ \overline{\alpha}N = \overline{\beta}N \imp
    d_{Y}(g(\alpha), g(\beta)) < \varepsilon \right].
  \]
  Let $x \in K$ and $u \in X$ such that $d_{X}(x,u) < 2^{-(N+1)}$.
  By item \ref{lem:StdRep4} of Lemma \ref{lem:StdRep}, there exists
  an $\alpha \in T_{1}$ and a $\beta \in T_{0}$ such that
  $\overline{\alpha}N = \overline{\beta}N$,
  $x = \Phi_{T_{0},X}(\alpha)$, and $u = \Phi_{T_{0},X}(\beta)$.
  Then $d_{Y}(f(x),f(u)) = d_{Y}(g(\alpha),g(\beta)) < \varepsilon$.
  Therefore $f$ is uniformly continuous near $K$. 
\end{proof}
We conclude with a summary of our main results.
\begin{theorem}\label{thm:EquivPi01FanM}
  The following statements are equivalent:
 \begin{enumerate}
   \item $\PiFTM$.
   \item $\scFT$.
   \item
   \textup{(\SCB)} Every pointwise continuous function from Baire space to $\Nat$
    is strongly continuous.
   \item
     \textup{(\SC)} Every pointwise continuous function from a
     complete separable metric space to a metric space is strongly
     continuous.
 \end{enumerate}
\end{theorem}

\section*{Acknowledgements}
I am grateful to Josef Berger and Makoto Fujiwara
for useful discussions.
I also thank Hajime Ishihara and Takako Nemoto for helpful comments.
This work was carried out while the author was  INdAM-COFUND-2012 fellow of
Istituto Nazionale di Alta Matematica ``F.\ Severi''(INdAM).


\end{document}